 \documentclass[numbook, envcountsect, envcountsame, envcountreset, runningheads, smallextended]{svjour3} 
\smartqed  
\usepackage{graphicx}
\usepackage{amssymb}
\usepackage{xcolor}
\usepackage{hyperref}

\begin{document}
\renewcommand*{\arraystretch}{1.3}

\noindent {\bf \large Stochastic dynamic programming with non-linear 
discounting }

\begin{center}{\bf  Nicole B\"auerle $^{a}$, Anna Ja\'skiewicz$^{b}$, Andrzej S. Nowak$^{c}$ } \end{center}
\noindent$^{a}$Department of Mathematics, Karlsruhe Institute of Technology, Karlsruhe, Germany\\ 
{\footnotesize {\it email: nicole.baeuerle@kit.edu}}\\
\noindent$^{b}$Faculty of Pure and Applied Mathematics, Wroc{\l}aw University of Technology,  Wroc{\l}aw, Poland \\
{\footnotesize {\it email: anna.jaskiewicz@pwr.edu.pl}}\\
\noindent$^{c}$Faculty of Mathematics, Computer Science, and Econometrics,  University of Zielona G\'ora,  Zielona G\'ora, Poland \\
{\footnotesize {\it email: a.nowak@wmie.uz.zgora.pl}}\\

\begin{abstract}

In this  paper, we study a Markov decision process with
a non-linear discount function and with a Borel state space. 
We define a recursive discounted utility, which resembles non-additive
utility functions considered in a number of models in economics. 
Non-additivity here follows from non-linearity of the discount function. 
Our study is complementary to the work of Ja{\'s}kiewicz, Matkowski
and Nowak \cite{jmn1} (Math. Oper. Res. 38 (2013),  108-121), 
where also non-linear discounting is used in the stochastic setting, but the expectation of utilities aggregated  on
the space of all histories of the process is applied leading to a 
non-stationary dynamic programming model. 
Our aim is to prove that in the recursive discounted utility case the 
Bellman equation has a solution
and there exists an optimal stationary policy for the problem in the infinite time horizon. 
Our approach includes  two cases: $(a)$ when the one-stage utility is bounded
on both sides by a weight function multiplied by some positive and negative constants, 
and $(b)$ when the  one-stage utility is unbounded  from below.


\end{abstract} 

\section{Introduction}
Discounting future benefits and costs is crucial in order 
to determine fair prices for investments and projects,
in particular, when long time horizons come into play like 
for example in the task to price carbon emissions. 
The idea behind pricing carbon emissions is 
to add up the cost of the economic damage caused by this emission for the society. 
Up to now the value of one ton $CO_2$ emissions 
varies quite heavily between countries, ranging for instance from $\$ 119.43$ 
in Sweden to $\$ 2.69$ in Japan (see \cite{carbon}). 
This is of course only partly due to the use of 
different discount functions, but nevertheless emphasises the role of discounting. 

The traditional discounting with a constant discount factor can be traced back to \cite{sam}. 
This is still the most common way to discount benefits and cost,  in particular because it simplifies the computation.  
However, Koopmans  \cite{k} gave an axiomatic characterisation 
of a class of recursive utilities, 
which also includes the classical way of discounting. 
He introduced an aggregator $W$ to aggregate the current utility $u_t$ with future ones 
$v_{t+1}$ in  $v_t=W(u_t,v_{t+1})$. When we choose $W(u,v)=u+\beta v,$ we get the classical discounting with a discount factor $\beta$.

In this  paper, we study a Markov decision process with a Borel state space, 
unbounded stage utility  and with a non-linear discount function $\delta$ which has certain properties. 
We use an aggregation of the form
$$ v_t = u_t + \int \delta\big(v_{t+1}(\cdot,x_{t+1})\big) q(x_{t+1}|x_t,\pi_t(h_t))$$
where  $q$ is the transition kernel of the Markov decision process, $\pi_t$ is the decision at time $t$ 
and $h_t$ the history of the process.  When $\delta(x)=\beta x$ we are back in the classical setting. 
In this case, it is well-known how to solve Markov decision process with an infinite time horizon, 
see for example \cite{br11,b,fs,hll,schal,slp}. 
In the unbounded utility case, the  established method is to use a weighted supremum norm 
and combine it with Banach's contraction theorem (see e.g. \cite{br11,b3b,b3,duran,hl}). 
In our setting the Banach contraction principle cannot be applied. 
Indeed our paper is in the spirit of \cite{jmn1,jmn2,jmn3}, where also non-linear discounting was used 
and an extension of the Banach theorem due to Matkowski \cite{m} was applied. 
Whereas  papers \cite{jmn2,jmn3} consider a purely deterministic decision process, work
\cite{jmn1} treats a stochastic problem. However, in \cite{jmn1}  the expectation is the final operator 
applied at the end of aggregation, whereas in the present paper expectation- and discounting-operators alternate. 
As will be explained in Section 3 this has the advantage that we get optimal stationary policies in our setting. 

The main result of our paper is a solution procedure of these new kind of discounting problems with stochastic transition. 
In particular, we provide an optimality equation, show that it has a solution
and prove that there exists an optimal stationary policy for the problem in the infinite time horizon. 
Note that we allow  the utility function  to be unbounded.

The outline of our paper is as follows. We introduce our model data together with the assumptions in Section 2. 
In Section 3, we present our optimisation problem. Particularly, we explain how the utility is aggregated in our model 
and what precisely the difference to  the model and results in \cite{jmn1} is. 
In Section 4, we summarise some auxiliary results like a measurable selection theorem and a generalised 
fixed point theorem, which is used later. 
Next, in Subsection 5.1 we treat the model, where the positive and negative part of the one-stage utility is bounded 
by a weight function $\omega.$ We show in this case that the value function $v^*$  
is a unique fixed point of the corresponding maximal reward operator and that every maximiser 
in the Bellman equation for $v^*$ defines an optimal stationary policy. 
In Subsection 5.2, we consider then the setting, where the utility function is unbounded from below, 
but still bounded from above by the weight function $\omega.$ Here, 
we can only show that the value function $\underline{v}^*$ is a fixed point  of the corresponding maximal reward operator, 
but examples show that the fixed point is not necessarily unique. 
Anyway, as in Section 5.1, any maximiser in the Bellman equation for
 $\underline{v}^*$  defines again an optimal stationary policy. 
 The proof employs an approximation of $\underline{v}^*$ by a monotone sequence of value functions, 
 which are bounded by a weight function $\omega$ in absolute value like in Subsection 5.1. 
 In Section 6, we briefly discuss two numerical algorithms for 
the solution of problems in Subsection 5.1, namely policy iteration and policy improvement. 
 The last section presents some applications. 
We discuss two different optimal growth models,  an inventory problem and a stopping problem.

\section{The dynamic programming model }

 Let  $\mathbb{N}$ ($\mathbb{R}$)  denote the set
of all positive integers (all real numbers)  and  $\underline{\mathbb{R}}=\mathbb{R}\cup\{-\infty\},$  $\mathbb{R}_+ =[0,\infty).$  
A {\it Borel space} $Y$ is a non-empty Borel subset of a complete
separable metric space.  By ${\cal B}(Y)$ 
we denote the $\sigma$-algebra of all Borel subsets of $Y $ 
and we write ${\cal M}(Y)$ to denote the set of all Borel measurable 
functions $g:Y\to \underline{\mathbb{R}}.$ 

A {\it discrete-time Markov decision process}  is specified by the following objects:
\begin{itemize}
\item[(i)] $X$ is the {\it state space} and is assumed to be a  Borel space.
\item[(ii)] $A$ is the {\it action space} and is assumed to be a  Borel
space.
\item[(iii)] $D$ is a non-empty Borel subset of $X\times A.$ We assume that for each $x\in X,$ the non-empty $x$-section
$$A(x):=\{a\in A: (x,a)\in D\}$$
of $D$ represents the set of {\it actions available in state} $x.$ 
\item[(iv)] $q$ is a  {\it transition probability} from $D$ to $X.$ For each $B \in {\cal B}(X)$, $q(B|x,a)$ 
is the probability that the new state is in the  set $B$, given the current state is 
$x\in X$ and an action $a\in A(x)$ has been chosen.
\item[(v)] $u\in {\cal M}(D)$ is a {\it one-period utility function.}
\item[(vi)] $\delta: \underline{\mathbb{R}}\to \underline{\mathbb{R}}$ is a {\it discount function.}
\end{itemize}

Let $D^{n}:= D\times\cdots\times D$ ($n$ times)  for $n\in \mathbb{N}.$ 
Let $H_1=X$ and $H_{n+1}$ be the space of {\it admissible histories} up to the $n$-th transition, i.e.,
$H_{n+1}:= D^n\times X$  for $n\in \mathbb{N}.$
An element of $H_n$ is called a partial history of the  process.
We put $H=D\times D\times\cdots $
and assume that $H_n$ and $H$ are equipped with the product
 $\sigma$-algebras.

In this paper, we restrict ourselves to deterministic policies, since randomisation
does not give any advantage from the point of view of utility maximisation.
 A {\it  policy} $\pi$ is a sequence $(\pi_n)$ of decision rules where,
for  every $n\in\mathbb{N},$ $\pi_n$  is a
Borel  measurable mapping, which associates any admissible history 
$h_n\in H_n$  $(n\in \mathbb{N})$ with an action
$a_n\in A(x_n).$  
We write $\Pi$ to denote the {\it set of
all policies}. Let $F$ be the set of all Borel measurable mappings $f:X\to A$ such that 
$f(x)\in A(x)$ for every $x\in X.$ 
When $A(x)$ is compact for each $x\in X$, then 
from  the Arsenin-Kunugui result (see Theorem
 18.18 in \cite{kech}), it follows that 
$F\not=\emptyset.$
   A  policy $\pi$ is called {\it stationary} if $\pi_n=f$ for all $n\in
	\mathbb{N}$ and some 
 $f\in F.$ Therefore, a stationary policy $\pi=(f,f,...)$ will be identified with  $f\in F$
 and the {\it set of all stationary policies} will be denoted by $F.$ 

\subsection{\bf Assumptions with comments}

Let 
 $\omega:X\to[1, \infty)$  be a fixed Borel measurable function.\\

\noindent
{\bf Assumptions (A):} \\
$(A2.1)$  there exists $b>0$ such that
$$u(x,a)\ge -b\omega (x)\quad \mbox{ for all } (x,a)\in D,$$
$(A2.2)$ there exists $c>0$ such that
$$u(x,a)\le c\omega (x)\quad \mbox{ for all } (x,a)\in D.$$\\

Our next assumption is on the discount function $\delta.$\\

\noindent
{\bf Assumptions (B):}\\
$(B2.1)$ there exists an   increasing function
$\gamma:\mathbb{R}_+ \to \mathbb{R}_+$ such that
$\gamma(z)<z$ for each $z>0$  and
   $$           |\delta(z_1)-\delta(z_2)|\le \gamma(|z_1-z_2|)$$
for all $z_1, z_2 \in \mathbb{R}.$\\
$(B2.2)$ $\delta$ is increasing, $\delta(0)=0$ and $\delta(-\infty)=-\infty$. \\ 
$(B2.3)$ $(i)$ $\gamma$ is subadditive, i.e., $\gamma(y+z)\le \gamma(y)+\gamma(z)$ 
for all $y,\ z\ge 0,$ and 

$\quad\;\; (ii)$ $\gamma(\omega(x)y)\le \omega(x)\gamma(y)$ for all $x\in X$ and $y> 0.$\\
$(B2.4)$ it holds that 
$$\int_X \omega(y)q(dy|x,a)\le \alpha\omega (x) \mbox{ for all } (x,a)\in D$$
$\quad$ and either 

$(i)$ $\alpha\le 1$ or 

$(ii)$  $\alpha>1$ and $\alpha\gamma(x)<x$ for all $x\in(0,+\infty).$\\

\begin{remark}
\label{rem1}
In some empirical studies it was observed that negative and  positive utilities were
discounted by different discount factors (``sign effect''). Therefore a simple non-linear discount function is $\delta(y) = \delta_1y$ for $y\le 0$
and $\delta(y) = \delta_2y$ for $y>0,$ where 
$\delta_1, \  \delta_2\in (0,1)$ and   $\delta_1 \ \not= \delta_2.$
For a discussion and interpretation of this and other types of discount functions the reader is referred to 
Ja\'skiewicz,  Matkowski and Nowak  \cite{jmn2}. 
Additional examples of discount functions are also  given in Section \ref{sec:appl}. 
\end{remark}
\begin{remark}
\label{rem2}
Obviously,  assumption $(B2.1)$ implies that $\gamma(0)=0 $
and $\gamma$ is continuous at zero.  
Moreover, it implies that $\delta$ is continuous and (with $(B2.2)$) 
that $|\delta(z)|\le \gamma(|z|)$ for all $z\in\mathbb{R}$.  From $(B2.3),$  it follows that
\begin{equation}
\label{niersubadd}
|\gamma(y)-\gamma(z)|\le \gamma(|y-z|)\quad\mbox{for all}\quad 
 y,\ z\ge 0.
\end{equation}
This fact and continuity of $\gamma$ at zero, implies that $\gamma$ is continuous at any point in $\mathbb{R}_+.$ 
\end{remark}

\begin{remark}
\label{remincr}
Assumption $(B2.3)$  holds if the function $z\mapsto \gamma(z)/z$ is non-increasing on $(0,\infty).$ Note that under this condition we have
$\frac{\gamma(y+z)}{y+z} \le \frac{\gamma(y )}{y}$ and 
$\frac{\gamma(y+z)}{y+z}\le \frac{\gamma( z)}{ z }.$ Hence, $\gamma(y+z)\le \gamma(y)+\gamma(z).$ 
Moreover, for any $d\ge 1  $ and $z>0,$ $\gamma(dz)/dz \le \gamma(z)/z.$ Thus
$\gamma(dz)\le d\gamma(z).$ 
Take $d=\omega(x).$ 
Then $(ii)$ in $(B2.3)$ holds.  
\end{remark}

\begin{remark}
\label{nonincrgam}
There are subadditive functions $\gamma$ such that $z\mapsto \gamma(z)/z$ is
not necessarily non-increasing. 
An example of such a subadditive  function is
$\gamma(x)=(1-\varepsilon)x+\varepsilon |\sin x|$ for some $\varepsilon\in (0,1).$
\end{remark}

The following two standard sets of assumptions  will be used alternatively.\\

\noindent
{\bf Assumptions (W):} \\
$(W2.1)$    $A(x)$ is compact for every $x\in X$ and the set-valued mapping $x\mapsto A(x)$ is upper semicontinuous, i.e.,
$\{x\in X:\; A(x)\cap K \not=
\emptyset\}$ is closed for each closed set
$K\subset A,$ \\
$(W2.2)$ the function $u$ is  upper semicontinuous
on $D,$\\
$(W2.3)$ the transition probability  $q$ is weakly continuous, i.e.,
$$(x,a) \mapsto\int_X \phi(y)q(dy|x,a)$$
is continuous  on $ D$ for each bounded continuous function $\phi$, \\
$(W2.4)$  the function $\omega$ is continuous on $X$.\\
$(W2.5)$ the function
$$(x,a) \mapsto\int_X \omega(y)q(dy|x,a)$$
is continuous  on $ D.$\\

\noindent
{\bf Assumptions (S):} \\
$(S2.1)$ $A(x)$ is compact for every $x\in X$, \\
$(S2.2)$ the function $u(x,\cdot)$ is  upper semicontinuous
on $A(x)$ for every $x\in X,$\\
$(S2.3)$ for each $x \in X$ and every Borel set $\widetilde{X} \subset X,$
the function $q(\widetilde{X}|x,\cdot)$ is continuous on $A(x),$ \\
$(S2.4)$ the function $a\mapsto\int_X \omega(y)q(dy|x,a)$ is continuous on $A(x)$ for every $x\in X.$\\

The above conditions were used in stochastic dynamic programming 
by many authors, see, e.g., 
Sch{\"a}l \cite{schal},  B\"auerle and  Rieder 
\cite{br11},  Bertsekas and Shreve \cite{bs} or Hern\'andez-Lerma and 
Lasserre  \cite{hll,hl}. Using the so-called ``weight'' or ``bounding''  
function $\omega$ one 
can study dynamic programming models with unbounded one-stage utility $u.$ 
This method was introduced by Wessels 
\cite{wjmaa}, but as noted by van der Wal  
\cite{vdw}, in the dynamic programming with linear discount function $\delta(z)=\beta z$, 
one can introduce an extra  state $x_e\not\in X$, re-define the transition probability 
and the utility function to obtain an equivalent  ``bounded model''. 
More precisely, we consider a new state space $X\cup\{x_e\},$ 
where $x_e$ is an absorbing isolated state. 
Let $A(x_e)=\{a_e\}$ with an extra action 
$a_e\not\in A.$ For $x\in X$ the action sets are $A(x)$. The transition probabilities $Q$
and one-stage utilities $R$ in a new model are as follows
$$Q(B|x,a):=\frac 1{\alpha \omega(x)}\int_B\omega(y)q(dy|x,a), 
\mbox{ for } B\in{\cal} B(X),\quad Q(x_e|x,a):=1-\frac{\int_X\omega(y)q(dy|x,a)}{\alpha\omega(x)},$$
$$R(x,a):=\frac{u(x,a)}{\omega(x)} , \mbox{ for }  (x,a)\in D,\quad R(x_e,a_e):=0.$$
Here, $\alpha$ is a constant from assumption $(B2.4).$
This transformed Markov decision process is equivalent to the original 
one in the sense that every policy gives the 
same total expected discounted payoff up to the factor 
$\omega (x)$, where $x\in X$ denotes the initial state.
We would like to emphasise that in the non-linear discount function 
case such a transformation to bounded case is not possible. We need to do some extra work.

\begin{remark}
\label{ubd} If $u$ is bounded from above by some constant, then we can skip assumptions $(A)$
and it is enough in assumptions $(B)$ to require 
$(B2.1),$ $(B2.2)$ and $(B2.3)(i)$.
In this case, it suffices to put  $\alpha =1,$ $\omega(x)=1$ for all $x\in X$ and 
it is easily seen that
$(B2.3)(ii)$, $(B2.4)$,  $(S2.4)$, $(W2.4)$ and $(W2.5)$ hold.  
If, on the other hand, $u$ is unbounded  
in the sense that there exists a function
$\omega$ meeting  conditions $(A2.2),$ then 
$\omega(x)\ge 1$ for all $x\in X$ must be unbounded as well. 
From Remark \ref{remincr}, it follows that 
$(B2.3)$ holds when the function  $z\mapsto \gamma(z)/z$ is non-increasing.    
We would like to emphasise that condition $(ii)$ in $(B2.3)$ 
is crucial in our proofs in the case when $(A2.2)$ holds  with 
an unbounded function $\omega.$ 
The dynamic programming problems when
only $(A2.2)$ is assumed  can be solved by 
a ``truncation method'' and  then making use of an approximation 
by solutions for models that satisfy conditions $(A).$
\end{remark}

\section{\bf Discounted utility evaluations: two alternative approaches}

Let $r_1(x,a)=u(x,a)$ for  $(x,a)\in D $ and, for any $n\in \mathbb{N} $ 
and $(h_{n},a_n)= 
(x_1,a_1,\ldots,x_n,a_n)\in D^n,$ 
\begin{eqnarray*}
\label{wzor1}
 & & r_{n+1}(h_{n+1},a_{n+1})\\   
&=& u(x_1,a_1)+\delta\big(u(x_2,a_2)+\delta\big(u(x_3,a_3)+\cdots+\delta(u(x_{n+1},a_{n+1})\big)\cdots)\big)\\  
 &=&u(x_1,a_1)+\delta(r_{n}(x_2,a_2,\ldots,x_{n+1},a_{n+1})).  
\end{eqnarray*}
Below in this subsection we assume that all expectations (integrals) and limits exist. 
In the sequel, we shall study cases where this assumption is satisfied.

Let $\pi\in\Pi $ and $x=x_1\in X$ be an initial state. By  $\mathbb{E}_x^\pi$ 
we denote the  expectation operator with respect to the unique probability measure 
$\mathbb{P}^\pi_x$ on $H$ induced 
by the  policy $\pi\in\Pi$ and the transition probability  $q$ according 
to the Ionescu-Tulcea theorem, see Proposition 7.28 in \cite{bs}.

\begin{definition}
\label{du} 
For any $\pi=(\pi_k) \in\Pi$  and any initial state $x=x_1$   
the $n$-stage expected discounted utility  is 
$$ R_n(x,\pi) = \mathbb{E}_x^\pi\big[r_n(x_1,a_1,...,x_n,a_n)\big] $$
and the expected discounted utility over an infinite horizon is 
\begin{equation}
\label{duinf}
R(x,\pi):= \lim\limits_{n\to\infty} R_n(x,\pi).
\end{equation}
A policy $\pi^*\in \Pi$ is optimal in the dynamic programming  model under
utility evaluation (\ref{duinf}), if
$$R(x,\pi^*) \ge R(x,\pi)\quad\mbox{  for all}\quad  \pi\in\Pi,\ x\in X.$$
\end{definition}

\begin{remark}
\label{remdu} Utility functions as in (\ref{duinf}) have been  considered by 
Ja\'skiewicz, Matkowski and Nowak \cite{jmn1}. Optimal policies have been  shown to exist
for the model with   $u$ {\it bounded from above} 
satisfying   assumptions (A) and either  (W) or (S).  However, optimal policies
obtained in \cite{jmn1}  are 
{\it history-dependent} and are characterised by an infinite system of
Bellman equations as in the non-stationary model of Hinderer \cite{h}.  
\end{remark}

To obtain {\it stationary optimal policies}, we shall define a {\it recursive
discounted utility} using the ideas similar to those developed in  papers on
dynamic programming  by Denardo \cite{d} and Bertsekas \cite{ber} and in papers
on economic dynamic optimisation \cite{balbus,jet,b3b,b3,duran,miao,streu,weil}. 
The seminal article for these studies was the work by Koopmans
\cite{k} on stationary recursive utility generalising the standard discounted utility of Samuelson \cite{sam}. 
To define the recursive discounted utility we must introduce some operator notation.

Let $\pi=(\pi_1,\pi_2,...)\in \Pi$ and   $v\in {\cal M}(H_{k+1})$. We set
$$Q_{\pi_k}^\delta v(h_k)=\int_X\delta(v(h_{k},\pi_k(h_k),x_{k+1}))q(dx_{k+1}|x_k,
\pi_k(h_k)) $$
and 
\begin{eqnarray*}
T_{\pi_k}v(h_k)&=&u(x_k,\pi_k(h_k)) +Q_{\pi_k}^\delta v(h_k)\\
 &=&u(x_k,\pi_k(h_k)) +\int_X\delta(v(h_{k},\pi_k(h_k),x_{k+1}))q(dx_{k+1}|x_k,
\pi_k(h_k)).
\end{eqnarray*}
These operators are well-defined for example when $u$ and $v$ are bounded from above.

Similarly, we define $Q_{\pi_k}^\gamma$ with $\delta$ replaced by $\gamma$.
Observe that by $(B2.1)$
$$
|Q_{\pi_k}^\delta v(h_k)|\le Q_{\pi_k}^\gamma |v|(h_k)
$$
provided that $Q_{\pi_k}^\delta v(h_k)>-\infty.$  
This fact will be used frequently.

The interpretation of $T_{\pi_k}v(h_k)$ is as follows. If 
$ x_{k+1} \mapsto v(h_{k},\pi_k(h_k),x_{k+1})$ is a ``continuation value'' of utility, then $T_{\pi_k}v(h_k)$ is the
expected discounted utility given the pair $(x_k,\pi_k(h_k)).$

The composition $T_{\pi_1}\circ T_{\pi_2}\circ\cdots \circ  T_{\pi_n} $ of the operators
$T_{\pi_1},T_{\pi_2},\ldots, T_{\pi_n}$ is for convenience denoted by
$T_{\pi_1}T_{\pi_2}\cdots T_{\pi_n}. $

Let $\bf 0$ be a function that assigns zero to each argument $y\in X.$ 

\begin{definition}
For any $\pi=(\pi_k) \in\Pi$  and any initial state $x=x_1,$ 
the $n$-stage recursive discounted utility  is defined as  
$$ U_n(x,\pi)=  T_{\pi_1} \cdots T_{\pi_n}{\bf 0}(x) $$
 and the recursive   discounted utility over an infinite horizon is 
\begin{equation}
\label{recduinf}
U(x,\pi):= \lim\limits_{n\to\infty} U_n(x,\pi).
\end{equation}
A policy $\pi^*\in \Pi$ is optimal in the dynamic programming  model under
utility evaluation (\ref{recduinf}), if
\begin{equation}\label{eq:problem3} U(x,\pi^*) \ge U(x,\pi)\quad\mbox{  for all}\quad  \pi\in\Pi,\ x\in X.
\end{equation}
\end{definition}

For  instance, below we give a full formula for $n=3.$ Namely,
\begin{eqnarray*}
& & U_3(x,\pi)=T_{\pi_{ 1}}T_{\pi_{2}}T_{\pi_3}{\bf 0}(x)
= u(x,\pi_1(x)) 
+\int_X\delta \Big(u(x_{2},\pi_{2}(x,\pi_1(x),x_{2}))\\
& +&  
  \int_X\delta(u(x_{3},\pi_3(x,\pi_1(x),x_{2},
\pi_{2}(x,\pi_1(x),x_{2}),x_3))q(dx_{3}|x_{2},
\pi_{2}(x,\pi_1(x),x_{2})))\Big)\\
& & q(dx_2|x,\pi_1(x)).
\end{eqnarray*}

We would like to point  out that in the special case of linear discount function 
$\delta(z)=\beta z$ with $\beta\in(0,1),$ the two above-mentioned approaches
coincide. In that case we deal with the usual expected  discounted utility, because
$$ R_n(x,\pi)= U_n(x,\pi) = \mathbb{E}_x^\pi\left[\sum_{k=1}^n \beta^k u(x_k,a_k)\right].$$

 \section{ Auxiliary results}

Let $Y$ be a Borel space.
By ${\cal U}(Y)$ we denote the space of all upper semicontinuous functions on $Y.$ 
We recall some results on measurable selections, a generalisation of the Banach fixed point theorem 
and present a property of a subadditive function $\gamma.$
 
\begin{lemma}
\label{lem1}  Assume that $A(x)$ is compact for each $x\in X.$ \\
 $(a)$ Let   $g\in {\cal M}(D)$ be  
 such that $a\mapsto g(x,a)$ is upper
semicontinuous on $A(x)$ for each $x\in X.$ Then,
$$ g^*(x):= \max\limits_{a\in A(x)}g(x,a)$$
is Borel measurable and there exists a Borel measurable mapping $f^*:X\to A$
such that $$f^*(x)\in {\rm Arg}\max_{a\in A(x)}g(x,a)
$$
for all $x\in X.$ \\
$(b)$ If, in addition, we assume that $x\mapsto
A(x)$ is upper semicontinuous and $g\in {\cal U}(D),$ then $g^*\in {\cal U}(X).$ 
\end{lemma}

Part $(a)$  follows from Corollary 1 in \cite{bp}. Part 
$(b)$ is a corollary to Berge's maximum
theorem, see pages 115-116 in \cite{berge} and Proposition 10.2 in \cite{schal}.\\

Let  ${\cal M}^a_b(X)$ be the space of all
functions   $ v\in {\cal M}(X) $  such that 
$x\mapsto v(x)/\omega(x)$ is bounded from above on $X$.
The symbol ${\cal M}^d_b(X)$ is used for the subspace of functions   
$v \in {\cal M}^a_b(X) $ such that $x\mapsto |v(x)|/\omega(x)$ is bounded
on $X.$ 
Let ${\cal M}^a_b(D)$ be the space of all
functions   $ w\in {\cal M}(D) $  such that
$(x,a)\mapsto w(x,a)/\omega(x)$ is bounded from above  on $D.$ 
By ${\cal M}^d_b(D)$ we denote the space of all 
functions   $ w\in {\cal M}_b^a(D) $  such that 
$(x,a)\mapsto |w(x,a)|/\omega(x)$ is bounded  on $D.$ We also define
	$$
	{\cal U}^a_b(X):={\cal M}^a_b(X)\cap {\cal U}(X),\quad
	{\cal U}^d_b(X):={\cal M}^d_b(X)\cap {\cal U}(X),\quad\mbox{and}
	$$
	$$
     {\cal U}^a_b(D):={\cal M}^a_b(D)\cap {\cal U}(D),\quad
	{\cal U}^d_b(D):={\cal M}^d_b(D)\cap {\cal U}(D).	
	$$

\begin{lemma}
\label{lematfin}
Let assumptions $(B)$ be satisfied and 
$$ I(v)(x,a):= \int_X\delta(v(y))q(dy|x,a), \quad 
v\in {\cal M}^a_b(X),\ (x,a)\in D.$$
Then $I(v) \in {\cal M}^a_b(D)$. If $v\in {\cal M}^d_b(X)$, then
$I(v)\in   {\cal M}^d_b(D).$ 
\end{lemma}
\begin{proof} 
Let $v\in {\cal M}^a_b(X)$ and $v^+(y)=\max\{v(y),0\}.$ 
Then there exists $c_0>0$ such that $v(y)\le v^+(y) \le c_0\omega(y)$ for all
$y\in X.$ Obviously, we have $I(v)(x,a)\le I(v^+)(x,a) \le I(c_0\omega)(x,a)
\le \alpha c_0\omega(x)$ for all $(x,a)\in D.$
Hence $I(v) \in  {\cal M}^a_b(D).$ Now assume that
 $v\in {\cal M}^d_b(X).$ Then there exists a constant $c_1>0$ such that $|v(y)|\le  c_1\omega(y)$ for all
$y\in X $ and we obtain
$|I(v)(x,a)| \le \alpha c_1\omega(x)$  for all $(x,a)\in D.$ 
Thus  $I(v) \in  {\cal M}^d_b(D).$
 \hfill $\Box$
\end{proof}

Our results will be formulated using the  standard dynamic programming operators.  
For any $v\in {\cal M}^a_b(X),$ put
$$ Sv(x,a):= u(x,a)+\int_X\delta(v(y))q(dy|x,a),\quad (x,a)\in D.$$ 
Next define 
\begin{equation}\label{T}
Tv(x):=\sup_{a\in A(x)} Sv(x,a) 
=\sup_{a\in A(x)}\left[u(x,a)+\int_X \delta(v(y))q(dy|x,a)\right].
\end{equation}
By $T^{(m)}$ we denote the composition of $T$ with itself $m$ times.

If $f\in F$ and $v\in {\cal M}^a_b(X), $ then we put 
\begin{equation}\label{Tf}
T_fv(x):=Sv(x,f(x))=u(x,f(x))+\int_X \delta(v(y))q(dy|x,f(x)).
\end{equation}
Clearly, $T_fv\in {\cal M}^a_b(X).$\\

The next   result  follows  from Lemmas \ref{lem1}-\ref{lematfin} and
Lemmas 8.3.7 and 8.5.5 from \cite{hl}. 

\begin{lemma}
\label{lem2}
Assume that assumptions $(A)$ and $(B)$ hold.\\
$(a)$ 
If  conditions $(W)$ are also satisfied and 
$v\in {\cal U}^d_b(X), $ 
 then  
 $Sv \in {\cal U}^d_b(D)$  and 
$Tv \in {\cal U}^d_b(X). $   \\
$(b)$  If    $(S2.2)$-$(S2.4) $ 
 hold and 
$v\in {\cal M}^d_b(X), $   
then  
 $Sv \in {\cal M}^d_b(D) $ and, for each $x\in X,$ the function
$a\mapsto Sv(x,a)$ is upper semicontinuous on $A(x).$ Moreover,
$Tv \in {\cal U}^d_b(X). $  
\end{lemma}

\begin{remark}
\label{optf}
$(a)$ The assumption that $\delta$ is continuous and increasing is important for part $(a)$ of Lemma \ref{lem2}.\\
$(b)$  Under assumptions of Lemma \ref{lem2},  in the operator in (\ref{T}) one can replace $\sup$ 
 by $\max $.\\ 
$(c)$ Using Lemma \ref{lematfin}, one can easily see that if $v\in
{\cal M}^a_b(X)$ and $f\in F,$ then $T_fv \in {\cal M}^a_b(X).$ 
 \end{remark}

  The following fixed point theorem 
will play an important role in our proof (see e.g.  \cite{m} or Theorem 5.2 in \cite{dg}).

\begin{lemma}\label{lem4}
Let $(Z,m)$ be a complete metric space, 
$\psi :\mathbb{R}_+ \to \mathbb{R}_+$ be a continuous, increasing function with $\psi(x)<x$ for all $x\in (0,\infty)$. 
If an operator $T:Z\to Z$ satisfies the inequality
$$ m(Tv_1,Tv_2)\le \psi(m(v_1,v_2))$$
for all $v_1,\ v_2\in Z$, then $T$ has a unique fixed point $v^*\in Z$ and $$ \lim_{n\to \infty}m(T^{(n)}v,v^*)=0$$
for each $v\in  Z $. Here $T^{(n)}$ is the composition of $T$ with itself $n$ times.
\end{lemma}

For the convenience of the reader we formulate and prove 
a modification of  Lemma 8 from  \cite{jmn1} that is used 
many times in our proofs. 
Consider a  function $\psi:\mathbb{R}_+ \to \mathbb{R}_+$ and put  
$$
 \psi_m(z)=z+\psi\big(z+\psi\big(z+\cdots+\psi(z+\psi(z))\cdots\big) \big), 
\ \mbox{ where } \ z>0 \mbox{ appears $m$ times.}
$$

\begin{lemma} \label{ap1} If $\psi$ is increasing,  subadditive and $\psi(y)< y$
for all $y>0,$ then for any $z>0$, there exists 
$$L(z):=\lim_{m\to\infty} \psi_{m}(z)= \sup_{m\ge 1}\psi_{m}(z)<\infty.$$
\end{lemma}
 \begin{proof} For any $k\in\mathbb{N},$ let $\psi^{(k)}$ mean  the composition of $\psi$ with itself $k$ times.
 Note that since the function $\psi$ is increasing, then for each $m\ge 1,$ 
$$\psi_{m+1}(z)>\psi_m(z).$$
Hence, the sequence $(\psi_{m}(z))$ is increasing. 
We show that its limit is  finite. 
Indeed, observe that by the subadditivity of $\psi,$ we have 
$$\psi_2(z)-\psi_1(z)=z+\psi(z)-z\le \psi(z),\quad\mbox{and}$$
$$\psi_3(z)-\psi_2(z)=z+\psi(z+\psi(z))-z-\psi(z)\le
\psi^{(2)}(z).$$
By induction, we obtain
$$\psi_{m}(z)-\psi_{m-1}(z)\le \psi^{(m-1)}(z).$$
Let $\epsilon>0$ be fixed. Since $\psi^{(m)}(z)\to 0$ as $m\to\infty$, there exists
$m\ge 1$ such that
$$\psi_m(z)-\psi_{m-1}(z)< \epsilon-\psi(\epsilon).$$
Observe now that from subadditivity of $\psi$   (set $\gamma:=\psi$ in (\ref{niersubadd})), it follows that 
\begin{eqnarray*} 
\psi_{m+1}(z)-\psi_{m-1}(z)&=&
\psi_{m+1}(z)-\psi_{m}(z)+\psi_{m}(z)-\psi_{m-1}(z)\\
&\le& z+\psi(\psi_{m}(z))-z-\psi(\psi_{m-1}(z))+\epsilon-\psi(\epsilon)\\
&\le&\psi(\psi_{m}(z)-\psi_{m-1}(z))+\epsilon-\psi(\epsilon)\\
&\le& \psi(\epsilon-\psi(\epsilon))+\epsilon-\psi(\epsilon)<\psi(\epsilon)+\epsilon-\psi(\epsilon)=\epsilon.
\end{eqnarray*}
By induction, we can easily prove that 
$$\psi_{m+k}(z)-\psi_{m-1}(z)\le \epsilon $$
for all $k\ge 0.$
Hence, $\psi_{m+k}(z)\le \psi_{m-1}(z)+ \epsilon.$ Since
$\psi_{m-1}(z)$ is finite, it follows that $L(z)$ is finite. \hfill
$\Box$
\end{proof}


\section{ Stationary optimal policies in dynamic problems with  the recursive discounted  utilities }\label{optimal}

In this section, we   prove that if $u\in {\cal M}^a_b(D),$
assumptions $(B)$ hold and either conditions $(W)$ or $(S)$  are satisfied, then 
the recursive discounted utility functions (\ref{recduinf}) are well-defined and there exists an optimal  stationary policy.
Moreover, under assumptions $(W)$ ($(S)$),  
the value function $x\mapsto \sup_{\pi\in\Pi}U(x,\pi)$ belongs
to ${\cal U}^a_b(X) $  (${\cal M}^a_b(X)$).
The value function and an optimal policy will be characterised 
 via a single Bellman equation. 
 First we  shall study the case $u \in {\cal M}^d_b(D)$  
and  then apply an approximation technique to the unbounded from below
case.

\subsection{ \bf One-period utilities with bounds on both sides}  

Assume that ${\cal M}^d_b(X)$ is endowed with the so-called
weighted norm $\|\cdot\|_\omega $  defined as
$$\|v\|_\omega:= \sup_{x\in X} \frac{|v(x)|}{\omega(x)},\quad
v\in  {\cal M}^d_b(X).$$
Then ${\cal M}^d_b(X)$ is a Banach space and ${\cal U}^d_b(X)$ is a closed subset of  ${\cal M}^d_b(X),$ if $(W2.4)$ holds. 
The following theorem is the main result of this subsection. Its proof is split in different parts below.\\
  
\begin{theorem}
\label{thm1}
Suppose that assumptions $(A),$ $(B)$ hold and assumptions $(W)$ are satisfied. Then\\
$(a)$ the Bellman equation $Tv=v$ has a unique solution $v^*
\in  {\cal M}^d_b(X)$ and 
$$\lim_{m\to\infty}\|T^{(m)}{\bf 0}-v^*\|_\omega=0\quad\mbox{and}\quad
 v^*(x)=\sup_{\pi\in\Pi}U(x,\pi), \ x\in X,$$
$(b)$ there exists $ f^*\in F$ such that $T_{ f^*}v^*=v^*$
and $f^*$ is an optimal stationary policy for problem (\ref{eq:problem3}),\\
$(c)$ $v^* \in  {\cal U}^d_b(X).$\\
The points (a) and (b) also remain valid  under assumptions $(A)$, $(B)$ and  $(S).$
\end{theorem}
\begin{remark}
\label{review}
As already mentioned, in our approach we consider only deterministic strategies.
This is because the optimality results do not change, when
we take randomised strategies into account. Actually, 
we may examine a new model
in which the original action sets $A(x)$ are replaced by the set of 
probability measures $\Pr(A(x))$. 
Then, the Bellman  equation has a solution as in Theorem \ref{thm1}, 
but the supremum in (\ref{T})
is taken over the set $\Pr(A(x))$. However, 
due to our assumptions the maximum is also attained 
at a Dirac delta concentrated at some point from $A(x).$
Therefore, randomised strategies do not influence the results.
 \end{remark}
Since
condition $(B2.4)$ contains two cases, it is convenient to define a new function
$$\tilde{\gamma}(y):=\left\{ \begin{array}{c c}
\gamma(y), & \alpha\le 1\\
\alpha\gamma (y), & \alpha>1.
\end{array}\right.
$$
Clearly, $\tilde{\gamma}$ is subadditive.
Let $z=\max\{b,c\},$ the constants $b>0$ and $c>0$ come from $(A).$
Then $|u(x,a)|\le\omega(x) z $ for all $(x,a)\in D.$ 
 From $(B2.3)(ii),$ it follows
that
\begin{equation}
\label{omegaout}
\tilde{\gamma}(\omega(x)y)\le \omega(x)\tilde{\gamma}(y),\quad\mbox{for all }\quad x\in X,\ y\ge 0.
\end{equation} 
This inequality is frequently used in our proofs. 
Let 
$$
 \tilde{\gamma}_k(z)=z+\tilde{\gamma}\big(z+\tilde{\gamma}(z+\cdots+\tilde{\gamma}(z)\big)\cdots \big),
$$
where $z$  appears on the right-hand side $k$ times. 
Putting $\psi=\tilde{\gamma}$
in Lemma \ref{ap1}, we infer that
\begin{equation}
\label{tildeL}
 \tilde{L}(z):= \lim_{k\to\infty}\tilde{\gamma}_k(z)= \sup_{k\in \mathbb{N}} \tilde{\gamma}_k(z) <\infty.
\end{equation}
We point out that $\tilde{\gamma}^{(n)}$ is the $n$-th iteration of the function  $\tilde{\gamma}.$\\

We now prove that the recursive discounted utility (\ref{recduinf}) is well-defined.

\begin{lemma} \label{lem6} If $u \in {\cal M}^d_b(D)$ and  assumptions $(B)$ 
are satisfied, then $U(x,\pi):=\lim_{n\to\infty}  U_n(x,\pi)$ exists for any policy $\pi\in\Pi$ and any initial state $x\in X.$
Moreover, $U(\cdot,\pi)\in {\cal M}^d_b(X)$ and 
$$\lim_{n\to\infty} \| U(\cdot,\pi)- U_n(\cdot,\pi)\|_\omega=0.$$ 
\end{lemma}
\begin{proof} 
   We shall prove that 
$(U_n(\cdot,\pi))$ is a Cauchy sequence of functions in ${\cal M}^d_b(X)$ for each  policy $\pi\in\Pi.$
We claim that  
\begin{equation}\label{w2}
|U_{n+m}(x,\pi)-U_n(x,\pi)|\le
Q_{\pi_1}^\gamma \ldots Q_{\pi_{n-1}}^\gamma Q^\gamma_{\pi_n} |T_{\pi_{n+1}}\ldots T_{\pi_{n+m}}{\bf 0}|(x).
\end{equation}
Indeed, using assumptions $(B)$, we can conclude that  
\begin{eqnarray}\label{w1}
\nonumber
|U_{n+m}(x,\pi)-U_n(x,\pi)|&=&|T_{\pi_1}\cdots   T_{\pi_{n}} T_{\pi_{n+1}}\cdots  T_{\pi_{n+m}} {\bf 0}(x)-T_{\pi_1}\cdots T_{\pi_n} {\bf 0}(x)|\\
\nonumber
&=&|Q_{\pi_1}^\delta T_{\pi_2}\cdots T_{\pi_{n+m}} {\bf 0}(x) - Q_{\pi_1}^\delta T_{\pi_2}\cdots T_{\pi_{n}} {\bf 0}(x) |\\ \nonumber
&\le&
Q_{\pi_1}^\gamma |T_{\pi_2}\cdots T_{\pi_{n+m}} {\bf 0}-T_{\pi_2}\cdots T_{\pi_{n}} {\bf 0}|(x) \le \quad \mbox{(cont.)}\cdots \\  \nonumber
&\le&
Q_{\pi_1}^\gamma \cdots Q_{\pi_{n-1}}^\gamma |Q^\delta_{\pi_n} T_{\pi_{n+1}}\cdots T_{\pi_{n+m}}{\bf 0}|(x) \\ \nonumber
&\le&
Q_{\pi_1}^\gamma \cdots Q_{\pi_{n-1}}^\gamma Q^\gamma_{\pi_n} |T_{\pi_{n+1}}\cdots T_{\pi_{n+m}}{\bf 0}|(x).
\end{eqnarray}

Assume that $m=1$. Then, for any $h_{n+1}\in H_{n+1},$  we have 
$$
 |T_{\pi_{n+1}}{\bf 0}(h_{n+1})|= |u(x_{n+1},\pi_{n+1}(h_{n+1}))|\le \omega (x_{n+1})z.$$
Take $m=2$ and notice by $(B)$  that
\begin{eqnarray*}
 |T_{\pi_{n+1}}T_{\pi_{n+2}}{\bf 0}(h_{n+1})|&=& |u(x_{n+1},\pi_{n+1}(h_{n+1}))+\\ &&\int_X 
 \delta(u(x_{n+2},\pi_{n+2}(h_{n+1},x_{n+2})))q(dx_{n+2}|x_{n+1},\pi_{n+1}(h_{n+1}))| \\   &\le& 
 \omega (x_{n+1})z+\int_X \gamma(\omega(x_{n+2})z)q(dx_{n+2}|x_{n+1},\pi_{n+1}(h_{n+1}))\\&\le&
 \omega (x_{n+1})z+\int_X\omega(x_{n+2}) \gamma(z)q(dx_{n+2}|x_{n+1},\pi_{n+1}(h_{n+1}))\\ &\le&
 \omega (x_{n+1})(z+\alpha\gamma(z))\le \omega(x_{n+1})(z+\tilde{\gamma}(z)).
 \end{eqnarray*}
 For $m=3,$ it follows that
 \begin{eqnarray*}
 \lefteqn {|T_{\pi_{n+1}}T_{\pi_{n+2}}T_{\pi_{n+3}}{\bf 0}(h_{n+1})|= |u(x_{n+1},\pi_{n+1}(h_{n+1}))+}\\ &&\int_X 
 \delta(T_{\pi_{n+2}}T_{\pi_{n+3}} {\bf 0}(h_{n+1},\pi_{n+1}(h_{n+1}),x_{n+2} ))q(dx_{n+2}|x_{n+1},\pi_{n+1}(h_{n+1}))| \\   &\le& 
 \omega (x_{n+1})z+
 \int_X \gamma\left(\omega(x_{n+2})(z+\tilde{\gamma}(z))\right)
 q(dx_{n+2}|_{n+1},\pi_{n+1}(h_{n+1}))
 \\  &\le&
 \omega (x_{n+1})z+\int_X\omega(x_{n+2})  \gamma\left(z+\tilde{\gamma}(z)\right)q(dx_{n+2}|x_{n+1},\pi_{n+1}(h_{n+1}))\\ &\le&
 \omega (x_{n+1})z+\alpha\omega (x_{n+1})\gamma(z+\tilde{\gamma}(z))\le \omega (x_{n+1})(z+\tilde{\gamma}(z+\tilde{\gamma}(z))).
  \end{eqnarray*}
Continuing this way, for any $h_{n+1}\in H_{n+1},$ we obtain
\begin{eqnarray} \label{w3}
 |T_{\pi_{n+1}}\cdots T_{\pi_{n+m}}{\bf 0}(h_{n+1}) |&\le& \omega(x_{n+1})
\big(z+\tilde{\gamma}\big(z+\tilde{\gamma}(z+\cdots+\tilde{\gamma}(z+\tilde{\gamma}(z))\cdots) \big)\big)
 \nonumber \\&=&\omega(x_{n+1})\tilde{\gamma}_m\left( z\right),
\end{eqnarray}
where $z$ appears $m$ times on the right-hand side of inequality (\ref{w3}). 
By (\ref{tildeL}), $\tilde{\gamma}_m( z)< \tilde{L} ( z )<\infty.$
Combining (\ref{w2}) and (\ref{w3}) and making use of $(B2.4)$ and  (\ref{omegaout}), we conclude that
\begin{eqnarray*}
Q_{\pi_1}^\gamma \ldots Q_{\pi_{n-1}}^\gamma Q^\gamma_{\pi_n} \tilde{L}(z)\omega  (x)&\le& 
 Q_{\pi_1}^\gamma \ldots Q_{\pi_{n-1}}^\gamma
\gamma(\tilde{ L}(z)) \alpha \omega (x) =Q_{\pi_1}^\gamma \cdots Q_{\pi_{n-1}}^\gamma \tilde{\gamma}(\tilde{L}(z))\omega (x) ...\\
&\le&\tilde{\gamma}^{(n)}(\tilde{L}(z))\omega (x).
 \end{eqnarray*}
Consequently,
\begin{equation}
\label{cauchyseq}
\|U_{n+m}(\cdot,\pi)-U_n(\cdot,\pi)\|_\omega \le\tilde{\gamma}^{(n)} (\tilde{L}(z)).
\end{equation}
From  the proof of (\ref{w3}) we deduce that for any $n\in\mathbb{N}$ 
$$|U_n(x,\pi)|\le \omega(x)\tilde{\gamma}_n(z)\le \omega(x) \tilde{L}(z).$$
Therefore, for each $n\in\mathbb{N},$  $U_n(\cdot,\pi)\in
{\cal M}_b^d(X).$ From (\ref{cauchyseq}) it follows that  
$(U_n(x,\pi))$ is a Cauchy sequence in the Banach space ${\cal M}^d_b(X).$ 
\hfill 
$\Box$
\end{proof}

\begin{proof}{\it of Theorem \ref{thm1}}  Consider  first assumptions $(W)$.
By Lemma \ref{lem2}, $T$ maps ${\cal U}^d_b(X)$ into itself. 
 We show that $T$ has a fixed point in ${\cal U}^d_b(X).$
Let $v_1,\ v_2 \in {\cal U}^d_b(X).$ Then, under assumptions $(B)$ we obtain
\begin{eqnarray*}
&&|Tv_1(x)-Tv_2(x) |\le \sup_{a\in A(x)} \int_X\left| \delta\big(v_1(y)\big)-\delta\big(v_2(y)\big)\right|q(dy|x,a)  \\ 
\le
&&\sup_{a\in A(x)}  \int_X  \gamma\big(|v_1(y) - v_2(y)|\big)q(dy|x,a)   \\ 
\le
&&\sup_{a\in A(x)}\left| \int_X \gamma\big(\|v_1 - v_2\|\big)\omega(y)q(dy|x,a)\right| \\ \le
&& \alpha \gamma\big( \|v_1-v_2\|_\omega\big)\omega(x).
\end{eqnarray*}
Hence,
$$
\|Tv_1-Tv_2\|_\omega \le \tilde{\gamma}( \|v_1-v_2\|_\omega).
$$
Since the space ${\cal U}^d_b(X)$ endowed with the metric induced by the norm
$\|\cdot\|_\omega$ is   complete, by Lemma \ref{lem4}, 
there exists a unique 
$v^*\in {\cal U}^d_b(X)$ such that $v^*=Tv^*$  and 
$$\lim_{n\to\infty} \|T^{(n)}v-v^*\|_\omega=0\quad\mbox{ for any }\quad v\in {\cal U}^d_b(X).$$ 

By Lemma \ref{lem1} and the assumptions that  $\delta$ is increasing and continuous,  it follows that 
there exists $f^*\in F$  such that
$v^*=T_{f^*}v^*.$ We claim that
$$v^*(x)=U(x,f^*)=\lim_{n\to\infty}T^{(n)}_{f^*}v^*(x) \quad \mbox{ for all } x\in X.$$
The operator $T_{f^*}:{\cal M}^d_b(X)\to {\cal M}^d_b(X)$ also satisfies 
assumptions of Lemma \ref{lem4}. Thus  there is a unique function $\tilde{v}\in
{\cal M}^d_b(X)$   such that 
 $$\tilde{v}(x)=T_{f^*}\tilde{v}(x)=\lim_{n\to\infty} T_{f^*}^{(n)} h(x),\quad x\in X,$$
for any $h\in {\cal M}^d_b(X).$ Therefore, $\tilde{v}=v^*.$ 
Putting $h: ={\bf 0}$ we deduce  from Lemma \ref{lem6}  that
$$\lim_{n\to\infty} T_{f^*}^{(n)} {\bf 0}(x)=U(x,f^*),\quad x\in X.$$
In order to prove the optimality of $f^*$ note that  for any $a\in A(x)$ and $x\in X$,  it holds
$$v^*(x)\ge u(x,a)+\int_X \delta(v^*(y))q(dy|x,a).$$
Taking any policy $\pi=(\pi_n)$ and iterating the above inequality, we get
$$v^*(x)\ge T_{\pi_1}\cdots T_{\pi_{n}}v^*(x),\quad x\in X.$$
We now prove that 
$$\lim_{n\to\infty} T_{\pi_1}\cdots T_{\pi_{n}} v^*(x)=U(x,\pi), \quad x\in X.$$
With this end in view, we first consider the differences
\begin{eqnarray*}
|U_{n}(x,\pi)- T_{\pi_1}\cdots T_{\pi_{n}} v^*(x)|&=& | T_{\pi_1}\cdots T_{\pi_{n}}{\bf 0}(x) -
 T_{\pi_1}\cdots T_{\pi_{n}}v^*(x)|\\
 &\le&\omega(x) \tilde{\gamma}^{(n)}(\|v^*\|_\omega)
\to 0 \quad \mbox{ as } n\to\infty.
 \end{eqnarray*}
By Lemma \ref{lem6}, $U_{n}(x,\pi)\to U(x,\pi)$ for every $x\in X$ as
 $n \to\infty.$ Therefore, we  have that
 $$v^*(x)\ge \lim_{n\to\infty}  T_{\pi_1}\cdots T_{\pi_{n}} v^*(x)= U(x,\pi),\quad x\in X$$
and 
$$ \sup_{\pi\in\Pi} U(x,\pi)\ge U(x,f^*)=v^*(x)\ge \sup_{\pi\in\Pi} U(x,\pi),\quad  x\in X.$$
This implies that
$$  U(x,f^*)=v^*(x)= \sup_{\pi\in\Pi} U(x,\pi),\quad  x\in X,$$
which  finishes the proof under assumptions $(W)$. 
For assumptions $ (S)$ the proof proceeds along the same lines. By Lemma
\ref{lem2}, under $(S)$, $T: {\cal M}_b^d(X)\to {\cal M}_b^d(X)$. \hfill $\Box$
\end{proof}

\begin{remark}
\label{valiter}
Under assumptions of Theorem \ref{thm1}, the Bellman equation
has a unique solution and it is the optimal value function $v^*(x)=
\sup_{\pi\in\Pi} U(x,\pi).$
Moreover, it holds that
$$\lim_{n\to\infty}\|T^{(n)}{\bf 0} -v^*\|_\omega=0.$$
Obviously, $T^{(n)}{\bf 0}$ is the value function in the $n$-step dynamic programming problem.  
One can say that the {\it value iteration algorithm} works
and the iterations
$T^{(n)}{\bf 0}(x)$  converge  to $v^*(x)$
  for each  $x\in X.$ This convergence is uniform in $x\in X$ when the weight function  $	\omega$ is bounded. 
\end{remark}


\subsection{\bf  One-period utilities unbounded from below}

In this subsection we drop condition $(A2.1)$ and assume that 
there exists $c>0 $ such that $u(x,a)\le c\omega(x)$ for all $(x,a)\in D.$ 
In other words, $u\in {\cal M}^a_b(D).$ Here we obtain the following result which is shown in the remaining part of this subsection.\\ 

\begin{theorem}
\label{thm2}
Suppose that assumptions  $(A2.2),$ $(B)$and $(W)$  are satisfied. Then\\
$(a)$ the optimal value function
$$\underline{v}^*(x):=\sup_{\pi\in\Pi}U(x,\pi), \ x\in X,$$
 is a solution to the Bellman equation $Tv=v$ and $\underline{v}^* \in {\cal M}^a_b(X),$\\
$(b)$ there exists $ \tilde{f}\in F$  such that $T_{\tilde{f}}\underline{v}^*=\underline{v}^*$
and $\tilde{f}$ is an optimal stationary policy,\\
$(c)$ $\underline{v}^* \in {\cal U}^a_b(X).$\\
The points (a) and (b) also remain valid  under assumptions $(A2.2)$, $(B)$ and  $(S).$
\end{theorem}

\begin{remark}
\label{bda}
We shall prove that $\underline{v}^*$ is the limit of a non-increasing sequence 
of value functions
in  ``truncated models'', i.e. the models that satisfy $(A2.1)$ and $(A2.2)$. 
The convergence is monotone, 
but it is not uniform. The Bellman equation may have many  unbounded solutions.
\end{remark}

\begin{remark}
\label{newrem} 
The assumptions of Theorem \ref{thm2} do not guarantee uniqueness. 
An example is very simple. 
Assume that $X=\mathbb{N},$ $A=A(x)=\{a\},$ $u(x,a)=0$ for all $(x,a)\in D,$ 
and the process moves from state $x$ to $x+1$
with probability one. The discount function $\delta(x)=\beta x$ with $\beta \in (0,1).$ 
Clearly, $u$ satisfies assumption $(A2.2)$ with $\omega(x)=1,$ $c=1.$ 
Note that $v (x)= r/\beta^x$ is a solution to the Bellman equation $Tv=v$ for any $r\in\mathbb{R}.$ 
Clearly, $\underline{v}^*(x)= 0$   is one of them.
Actually,  $\underline{v}^*(x)= 0$ is 
the largest non-positive solution to the Bellman equation. 
This example does not contradict
the uniqueness result in Theorem \ref{thm1}.   
Within the class of bounded functions $\underline{v}^*(x)= 0$ is
 the unique solution to the Bellman equation.
\end{remark} 

We now prove that the recursive discounted utility (\ref{recduinf}) is well-defined.

\begin{lemma} \label{l3} If $u \in {\cal M}^a_b(D)$ and  assumptions
 $(B)$ 
are satisfied, then $U(x,\pi):=\lim_{n\to\infty}  U_n(x,\pi)$ 
exists in $\underline{\mathbb{R}}$ for any policy $\pi\in\Pi$ and any initial state $x\in X.$ Moreover,
 $U(\cdot,\pi)\in {\cal M}^a_b(X).$ 
\end{lemma}

Our assumption that $u\in {\cal M}^a_b(D)$ means that
$(A2.2)$ holds, i.e., there exists $c>0$ such that $u(x,a)\le c\omega(x)$ for all $(x,a)\in D.$ 

\begin{proof}{\it of Lemma \ref{l3}}  We divide the proof into five parts.

\noindent 
{\it Step 1}  We start with a simple observation: for any $n\in\mathbb{N},$  $x\in X$ and $\pi\in\Pi$ it holds
 $$U_{n+1}(x,\pi)\le U_n(x,\pi)+\tilde{\gamma}^{(n)}(c)\omega(x).$$
 From assumptions $(B),$ it follows that
$$\delta(a+c\omega(x))\le \delta (a)+ \gamma(c\omega(x)) 
\le \gamma(c)\omega(x)$$ 
for $a\in\underline{\mathbb{R}}$ and $x\in X.$
Note that, for any $h_k\in H_k$ and $\pi_k$, we have 
\begin{eqnarray}\label{krzywda3}
\nonumber
T_{\pi_k} c\omega(h_k)&=&u(x_k,\pi_k(h_k))+\int_X\delta (c\omega(y))q(dy|x_k,\pi_k(h_k))\\
&\le&u(x_k,\pi_k(h_k))+Q^\gamma_{\pi_k}c\omega(x_k)\le T_{\pi_k}{\bf 0}(h_k)+\tilde{\gamma}(c)\omega(x_k).
\end{eqnarray} 
Furthermore,  for any  $v\in {\cal M}(H_{k+1})$
such that $v(h_k,\pi_k(h_k),y)\le \eta\omega(y)$ for all $y\in X$ and some $\eta>0,$ 
 we obtain
\begin{eqnarray*}
T_{\pi_k} (v+c\omega)(h_k)&=&u(x_k,\pi_k(h_k))+\int_X\delta (v(h_k,\pi_k(h_k),y)+ c\omega(y))q(dy|x_k,\pi_k(h_k))\\
&\le& u(x_k,\pi_k(h_k))+Q^\delta_{\pi_k} v(h_k)+ Q^\gamma_{\pi_k} c\omega(h_k)  
\le T_{\pi_k} v(h_k)+\tilde{\gamma}(c)\omega(x_k).
\end{eqnarray*}
From this fact and (\ref{krzywda3}) we conclude that 
\begin{eqnarray}
\label{nnnn}
U_{n+1}(x,\pi)&=&T_{\pi_1}\cdots T_{\pi_n}T_{\pi_{n+1}}{\bf 0}(x)\le 
T_{\pi_1}\cdots T_{\pi_{n-1}}T_{\pi_n}c\omega (x)\\
&\le& T_{\pi_1}\cdots T_{\pi_{n-1}}\big(T_{\pi_n}{\bf 0} +
\tilde{\gamma}(c)\omega\big)(x) \nonumber \\
&\le&
 T_{\pi_1}\cdots T_{\pi_{n-2}}\big(T_{\pi_{n-1}}T_{\pi_n}{\bf 0} 
+\tilde{\gamma}^{(2)}(c)\omega\big)(x) 
 \ldots\  \mbox{(cont.)}  \nonumber \\
&\le& T_{\pi_1}\cdots T_{\pi_{n-2}}T_{\pi_{n-1}}T_{\pi_n}{\bf 0} (x)+\tilde{\gamma}^{(n)}(c)\omega(x)  \nonumber  \\
&=&U_n(x,\pi)+\tilde{\gamma}^{(n)}(c)\omega(x). \nonumber
\end{eqnarray}
This finishes the first step.\\
 
\noindent 
{\it Step 2} Let $U_{n}(x,\pi)=-\infty$ for some $n\in\mathbb{N}$, $x\in X$ and $\pi\in\Pi$. 
Then,  by Step 1, $U_{m}(x,\pi)=-\infty$ for all $m\ge n$. Therefore, $\lim_{m\to\infty}U_{m}(x,\pi)=-\infty.$\\

\noindent
{\it Step 3}
Let 
 $v \in {\cal M}(H_{k+1}) $  be such that $v(h_k,a_k,x_{k+1}) \le \eta\omega(x_{k+1})$
for every  $x_{k+1}\in X$ and for some $\eta>0.$ Define
$${\Gamma}_{\pi_k}v(h_k):=c\omega(x_k)+Q^\gamma_{\pi_k}v(h_k)\quad\mbox{and} $$
$$
\Gamma_{\pi_{n+1}}^{\pi_{n+m}}{\bf 0}(h_{n+1}):=  \Gamma_{\pi_{n+1}}\cdots
\Gamma_{\pi_{n+m}}{\bf 0}(h_{n+1}).
$$
Note that
$$
\Gamma_{\pi_{n+m}}{\bf 0}(h_{n+m}) = 
 c\omega(x_{n+m}).$$
Next, we have 
\begin{eqnarray*}
&&\Gamma_{\pi_{n+m-1}} \Gamma_{\pi_{n+m}}{\bf 0}(h_{n+m-1})\\
&=& c\omega(x_{n+m-1})+ \int_X\gamma\big(c\omega( x_{n+m})\big)
q\big(dx_{n+m}|x_{n+m-1},\pi_{n+m-1}(h_{n+m-1})\big)\\ &\le&
\omega(x_{n+m-1})(c+\tilde{\gamma}(c)),
\end{eqnarray*}
and 
\begin{eqnarray*}
&&\Gamma_{\pi_{n+m-2}}\Gamma_{\pi_{n+m-1}} \Gamma_{\pi_{n+m}}{\bf 0}(h_{n+m-2}) \\ &\le& 
c\omega(x_{n+m-2}) +
\int_X\gamma\big( \omega( x_{n+m-1})(c+\tilde{\gamma}(c))\big) 
q\big(dx_{n+m-1}|x_{n+m-2},\pi_{n+m-2}(h_{n+m-2})\big)\\ 
&\le& \omega(x_{n+m-2})(c+ \tilde{\gamma}(c+\tilde{\gamma}(c))).
\end{eqnarray*}
Continuing this way, we get
$$
\Gamma_{\pi_{n+1}}^{\pi_{n+m}}{\bf 0}(h_{n+1}) =  \Gamma_{\pi_{n+1}}\cdots
\Gamma_{\pi_{n+m}}{\bf 0}(h_{n+1}) \le
\omega(x_{n+1})
\big(c+\tilde{\gamma}\big(c+\tilde{\gamma}\big(c+\cdots+\tilde{\gamma}(c+\tilde{\gamma}(c))\cdots\big) \big)\big),$$
where $c$ appears on the right-hand side of this inequality $m$ times. Putting $\tilde{c} =\tilde{L}(z)$ with $z=c$ in
(\ref{tildeL}), we obtain
\begin{equation}
\label{GammaL}
\Gamma_{\pi_{n+1}}^{\pi_{n+m}}{\bf 0}(h_{n+1}) =  \Gamma_{\pi_{n+1}}\cdots
\Gamma_{\pi_{n+m}}{\bf 0}(h_{n+1}) \le \omega(x_{n+1})\tilde{c}=\omega(x_{n+1})\tilde{L}(c) <\infty.
\end{equation}\\

\noindent  {\it Step 4} 
For $m,\ n \in\mathbb{N},$  we set
$$W_{n,m}(x,\pi):=T_{\pi_1}\cdots T_{\pi_n}\Gamma_{\pi_{n+1}}^{\pi_{n+m}}  {\bf 0}(x)\quad\mbox{and}\quad 
W_{n,0}(x,\pi):=T_{\pi_1}\cdots T_{\pi_n}{\bf 0}(x)=U_n(x,\pi).
$$
For any $k=1,...,n-1$ and $h_{k+1}\in H_{k+1},$ $\pi\in\Pi,$ let 
$\pi(k+1)= (\pi_{k+1},\pi_{k+2},...)$ and
$$V^{\pi(k+1)}_{n-k,m}(h_{k+1}):= T_{\pi_{k+1}}\cdots T_{\pi_n}
\Gamma_{\pi_{n+1}}^{\pi_{n+m}}  {\bf 0}(h_{k+1}),$$
$$V^{\pi(k+1)}_{n-k,0}(h_{k+1}):= T_{\pi_{k+1}}\cdots T_{\pi_n}
  {\bf 0}(h_{k+1}).$$
For $k=n-1$, we have $$V^{\pi(n)}_{1,m}(h_n)=
T_{\pi_n}
\Gamma_{\pi_{n+1}}^{\pi_{n+m}}  {\bf 0}(h_n), 
$$
and
$$V^{\pi(n)}_{1,0}(h_n)=
T_{\pi_n}
 {\bf 0}(h_n)= u(x_n,\pi_n(h_n)).
$$
Hence and from (\ref{GammaL}), it follows that
\begin{equation}
\label{wwnn}
V^{\pi(n)}_{1,m}(h_n)- V^{\pi(n)}_{1,0}(h_n) \le
\int_X \gamma\big(
\omega(x_{n+1})\tilde{c}\big)q(dx_{n+1}|x_n,\pi_n(h_n)) \le
\omega(x_n)\tilde{\gamma}(\tilde{c}).
\end{equation}
Observe that for each $k= 1,...,n-2,$
\begin{eqnarray}\label{vkeq}
V^{\pi(k+1)}_{n-k,m}(h_{k+1}) - V^{\pi(k+1)}_{n-k,0}(h_{k+1})
	&=&	T_{\pi_{k+1}}V^{\pi(k+2)}_{n-k-1,m}(h_{k+1}) -
	T_{\pi_{k+1}}V^{\pi(k+2)}_{n-k-1,0}(h_{k+1}) \\ \nonumber 
	& =&
	Q^\delta_{\pi_{k+1}}V^{\pi(k+2)}_{n-k-1,m}(h_{k+1})-
	Q^\delta_{\pi_{k+1}}V^{\pi(k+2)}_{n-k-1,0}(h_{k+1})  
	\nonumber \\
	&\le& Q^\gamma_{\pi_{k+1}}\big(V^{\pi(k+2)}_{n-k-1,m} -
	 V^{\pi(k+2)}_{n-k-1,0}\big)(h_{k+1}). \nonumber
	\end{eqnarray}
	It is important to note that
	$$V^{\pi(k+1)}_{n-k ,m}-
	V^{\pi(k+1)}_{n-k ,0}> 0,\quad k=1,...,n-1.$$
Now, using (\ref{vkeq}), for any $\pi\in\Pi$ and $x=x_1,$  we 
conclude that
\begin{eqnarray*}
W_{n,m}(x,\pi)- W_{n,0}(x,\pi) &=&T_{\pi_1}V^{\pi(2)}_{n-1 ,m}(x)
-T_{\pi_1}V^{\pi(2)}_{n-1 ,0}(x)\\ &=&
Q^\delta_{\pi_1} V^{\pi(2)}_{n-1 ,m}(x) -Q^\delta_{\pi_1}V^{\pi(2)}_{n-1 ,0}(x) \\
&\le& Q^\gamma_{\pi_1}\big(
V^{\pi(2)}_{n-1 ,m}  - V^{\pi(2)}_{n-1 ,0}\big)(x) \\
&\le&  
Q^\gamma_{\pi_1}Q^\gamma_{\pi_2}
\big(
V^{\pi(3)}_{n-2 ,m}  - V^{\pi(3)}_{n-2 ,0}\big)(x)  ...\ \mbox{(cont.)}\\
&\le& Q^\gamma_{\pi_1}Q^\gamma_{\pi_2}\cdots Q^\gamma_{\pi_{n-1}}
\big(
V^{\pi(n)}_{1 ,m}  - V^{\pi(n)}_{1 ,0}\big)(x).
\end{eqnarray*}
This and (\ref{wwnn}) imply that
\begin{equation}
\label{wwwnn}
W_{n,m}(x,\pi)- W_{n,0}(x,\pi)\le \omega(x)\tilde{\gamma}^{(n)}(\tilde{c}),\ \mbox{for all}\ m,\ n \in \mathbb{N}. 
\end{equation}\\

\noindent
{\it Step 5} We now consider the case where 
  $U_{n}(x,\pi)>-\infty$ for an initial state $x\in X,$ a policy    $\pi\in\Pi$ and for all $n\in\mathbb{N}$. 
From (\ref{GammaL}) we have 
$$  
W_{n,m}(x,\pi)\le
W_{n,m+1}(x,\pi)
\le \Gamma_{\pi_1}^{\pi_{n+m+1}}{\bf 0} (x)\le \tilde{c}\omega(x) <\infty.$$
Therefore, $\lim_{m\to\infty} W_{n,m}(x,\pi)$ exists and is 
finite. Let us denote this limit by $G_n$. 
Note that, for each $m,\ n\in \mathbb{N},$
$$
U_n(x,\pi)= W_{n,0}(x,\pi) \le W_{n,m}(x,\pi).
$$
Let $\epsilon>0$ be fixed. Then, by (\ref{wwwnn}), for 
sufficiently large $n,$ say $n>N_0,$
$$ 
W_{n,m}(x,\pi)\le W_{n,0}(x,\pi)+\epsilon
$$
for all $m\in\mathbb{N}.$ Thus
$$
W_{n,m}(x,\pi)  -  \epsilon\le W_{n,0}(x,\pi)\le W_{n,m}(x,\pi)
$$
and consequently
$$
G_n  -  \epsilon\le W_{n,0}(x,\pi)\le G_n
$$
for all $n> N_0.$
Observe that the sequence $(G_n)$ is non-increasing and 
$G_*:=\lim_{n\to\infty}G_n$ exists
in the extended real line $\underline{\mathbb{R}}$.
Hence,  the limit
 $$\lim_{n\to\infty}W_{n,0}(x,\pi)=\lim_{n\to\infty} U_n(x,\pi)$$ also exists 
and equals $G_*.$\hfill
$\Box$
\end{proof}

In the proof of Theorem \ref{thm2}  we shall need the following result (see
 \cite{schal} or Theorem A.1.5 in \cite{br11}).

\begin{lemma}\label{sch} If $Y$ is a metric space and $(w_n)$ is a non-increasing 
sequence of upper semicontinuous functions 
 $w_n:Y\to \underline{\mathbb{R}},$ then\\
(a) $w_\infty=\lim_{n\to\infty}w_n$ exists and $w_\infty$ is upper 
semicontinuous,\\
(b)  if, additionally, $Y$ is compact, then 
$$\max_{y\in Y}
\lim_{n\to\infty}w_n(y)=\lim_{n\to\infty}\max_{y\in Y}w_n(y).$$
\end{lemma}

In the proof of Theorem \ref{thm2} we shall refer to the dynamic programming operators defined in (\ref{T}) and (\ref{Tf}).
Moreover, we also define corresponding operators for $v\in {\cal M}^a_b(X)$ and $K\in\mathbb{N}$ as follows
$$T^K v(x):=\sup_{a\in A(x)}\left[u^K(x,a)+\int_X\delta(v(y))q(dy|x,a)\right],\quad x\in X,$$
and
$$T_{f}^K v(x)=u^K(x,f(x))+\int_X \delta(v(y))q(dy|x,f(x)),\quad x\in X,$$
where $f\in F$ and $u^K(x,a)=\max\{u(x,a),1-K\},$ $K\in\mathbb{N}.$ 
The recursive discounted utility functions with one-period utility $u^K$  
in the finite ($n$-periods)  and infinite time horizon for an initial state $x\in X$ and a policy $\pi\in\Pi$  
will be denoted by  $U_n^K(x,\pi)$ and $U^K(x,\pi),$ respectively.
 
\begin{lemma}\label{aux} For any $n\in\mathbb{N}$ and $f\in F$, it holds
$$\lim_{K\to\infty} U^K_n(x,f)=U_n(x,f),\quad x\in X.$$
\end{lemma}
\begin{proof}  We proceed by induction. For $n=1$ the fact is obvious. Suppose that
$$T^{K,(n)}_f{\bf 0}(x)=U^K_n(x,f)\to T^{(n)}_f{\bf 0}(x)=U_n(x,f)\quad\mbox{as}\ 
 K\to\infty, \  \mbox{for all}\  x\in X.$$
 Here, $T^{K,(n)}_f$ denotes the $n$-th composition of the operator $T^{K}_f$ with itself.
Then, by our induction hypothesis, our assumption that $\delta$ is continuous and increasing, and
the monotone convergence theorem, we infer that, for every $x\in X,$
\begin{eqnarray*}
\lim_{K\to\infty}T_f^K T_f^{K,(n)}{\bf 0} (x)&=&
\lim_{K\to\infty}\left(u^K(x,f(x))+\int_X\delta (T_f^{K,(n)}{\bf 0} (y))q(dy|x,f(x)) \right)\\
&=& u(x,f(x))+\int_X \delta (T_f^{(n)}{\bf 0} (y))q(dy|x,f)= T_f^{(n+1)}{\bf 0} (x).
\end{eqnarray*}
The lemma now follows by the induction principle. \hfill $\Box$
\end{proof}

\begin{proof} {\it of Theorem \ref{thm2}} Assume first $(W)$. The proof for $(S)$ is analogous with obvious changes. 
By Theorem \ref{thm1}, for any $K\in\mathbb{N},$ there exists a unique solution  
$v^{*,K}\in {\cal U}^d_b(X)$ to the Bellman  equation and, for each  $x\in X,$ 
$v^{*,K}(x) = \sup_{\pi\in\Pi} U^K(x,\pi).$  Since, $u^K \ge u,$ it follows that
$$
v^{*,K}(x) = \sup_{\pi\in\Pi} U^K(x,\pi)\ge \sup_{\pi\in\Pi} U(x,\pi)=\underline{v}^*(x), \quad x\in X.
$$
Clearly, the sequence $(v^{*,K})$ is non-increasing, thus $v_\infty(x) := \lim_{K\to\infty} v^{*,K}(x)$ 
exists in $\underline{\mathbb{R}}$ for every $x\in X,$ and consequently, 
 \begin{equation}\label{nie1}
v_\infty(x) \ge \underline{v}^*(x), \quad x\in X.
\end{equation}

From Theorem \ref{thm1} we know that $v^{*,K}$ is a solution to the  equation 
$$ 
v^{*,K}(x)=\sup_{a\in A(x)}\left[u^K(x,a)+\int_X \delta(v^{*,K}(y))q(dy|x,a)\right], \quad x\in X.
$$
Since both sequences $(v^{*,K}(x)),$ $x\in X,$ and $(u^K(x,a)),$ $(x,a)\in D,$ are non-increasing, it follows from Lemma \ref{sch}, 
our assumption that $\delta$ is increasing and  continuous  and 
the monotone convergence  theorem    that
\begin{eqnarray}\label{oegr}\nonumber
v_\infty(x) &=&  \lim_{K\to\infty} v^{*,K}(x) =
 \lim_{K\to\infty}  \max_{a\in A(x)}\left[u^K(x,a)+\int_X \delta(v^{*,K}(y))q(dy|x,a)\right]\\\nonumber
&=&\max_{a\in A(x)}\lim_{K\to\infty}\left[u^K(x,a)+\int_X \delta(v^{*,K}(y))q(dy|x,a)\right]\\
&=& \max_{a\in A(x)}\left[u(x,a)+\int_X \delta(v_\infty(y))q(dy|x,a)\right],\quad x\in X.
\end{eqnarray}
Moreover, in case  $(W),$ we have that $v_\infty \in {\cal U}^a_b(X).$ 
From the obvious inequalities $u(x,a)\le u^1(x,a)\le c\omega(x),$ $(x,a)\in D,$
it follows that $v_\infty (x) \le\tilde{c}\omega(x)$ for $\tilde{c}= \tilde{L}(c) $ and for all $x\in X$ (put $z=c$ in   
(\ref{tildeL})).   By Lemma \ref{lem1}, there exists a maximiser $\tilde{f}\in F$ 
on the right-hand side of equation (\ref{oegr}) and we have 
$$v_\infty (x) = u(x,\tilde{f}(x))+\int_X \delta(v_\infty(y))q(dy|x,\tilde{f}(x))=T_{\tilde{f}} v_\infty(x),\quad x\in X.$$
Iterating this equation, we obtain that
$$v_\infty (x)= T^{(n)}_{\tilde{f}} v_\infty(x) \le T^{K,(n)}_{\tilde{f}}v_\infty(x)
\le T^{K,(n)}_{\tilde{f}} \tilde{c}\omega (x), \  \mbox{for all}\ x\in X\ \mbox{and }\ k\in\mathbb{N}.$$
From (\ref{nnnn}) in the proof of Lemma  \ref{l3},  
(with $c$ replaced by $\tilde{c},$  $u$ replaced by $u^K$ 
and $\pi_1=\cdots=\pi_n= \tilde{f}$) 
we infer  that 
$$v_\infty (x)\le T^{K,(n)}_{\tilde{f}}\tilde{c}\omega(x) \le U^K_n(x,\tilde{f})+\tilde{\gamma}^{(n)}(\tilde{c})\omega(x),\ 
\mbox{for all}\   x\in X\ \mbox{and } n\in \mathbb{N}.$$
Letting $K\to \infty$ in the above inequality and making use of Lemma \ref{aux} yield 
that $$v_\infty (x)\le U_n(x,\tilde{f})+\tilde{\gamma}^{(n)}(\tilde{c})\omega(x)\ \mbox{ for all}\ x\in X.$$
Hence,
$$v_\infty (x)\le \lim_{n\to\infty}\left(  U_n(x,\tilde{f})+\tilde{\gamma}^{(n)}(\tilde{c})\omega(x)
\right)= U(x,\tilde{f})\le \sup_{\pi\in\Pi} U(x,\pi)=\underline{v}^*(x),\quad x\in X.$$
From this inequality and (\ref{nie1}),  
 we conclude that
$$
v_\infty (x)= U(x,\tilde{f})= \sup_{\pi\in\Pi} U(x,\pi)= \underline{v}^*(x),\ \mbox{for all}\  x\in X,
$$
and the proof is finished.
\hfill $\Box$
\end{proof}

\section{Computational Issues}
In this section we consider the unbounded utility setting as in Theorem \ref{thm1}. 

\subsection{ \bf Policy iteration}
An optimal stationary  policy can be computed as a limit point of a sequence of decision rules.  
In what follows, we define  $V_0={\bf 0}$ and  $V_n := T^{(n)} {\bf 0}$ for $n\in\mathbb{N}$. Next for fixed $x\in X,$ let
$$ A_{n }^\ast(x) := 
\mbox{Arg}\max_{a\in A(x)} \Big(
u(x,a) + \int_X \delta(V_{n-1}(y))q(dy|x,a)\Big)$$ for $n\in\mathbb{N}. $  
In the same way, let
$$ A^\ast(x) := 
\mbox{Arg}\max_{a\in A(x)} \Big(
u(x,a) + \int_X \delta(v^*(y))q(dy|x,a)\Big).$$ 
By $ Ls A_n^\ast(x),$ we denote
the {\em upper limit of the set sequence}  $(A_n^\ast(x)),$ that is,
the set of all {\it accumulation points} 
of  sequences $ (a_n)$ with 
   $a_n\in A_n^\ast(x) $ for all $n\in\mathbb{N}.$ 

The next result states that an optimal stationary policy can be obtained from 
accumulation points of sequences of maximisers of  recursively computed value functions.  
Related results for dynamic programming with standard discounting are discussed, for example, in \cite{br11} and \cite{schal}.

\begin{theorem}\label{theo:policyiteration}
Under assumptions of Theorem \ref{thm1}, we obtain: $\emptyset \neq Ls A_n^\ast(x) \subset   A^\ast(x)$ for all $x\in{X}$.
\end{theorem}

\begin{proof}
Fix $x\in{X}$  and define for $n\in\mathbb{N}$ the functions $v_n :
A(x) \to\mathbb{R}$ by
$$v_n(a) := u(x,a) +  \int_X \delta(V_{n-1}(y))q(dy|x,a).$$
By Lemma \ref{lem2}, $v_n $ is upper semicontinuous on $A(x).$
Moreover, 
 for $m\ge n$ and $a\in A(x)$,  we have 
\begin{eqnarray}
\label{vmn}
 |v_m(a)-v_n(a)| &\le& |\int_X \delta\big( T^{(m-1)}{\bf 0}(y)\big)q(dy|x,a) -
\int_X \delta\big( T^{(n-1)}{\bf 0}(y)\big)q(dy|x,a) |\\ &\le&
\int_X \gamma\big( |T^{(m-1)}{\bf 0}(y)-T^{(n-1)}{\bf 0}(y)|\big)q(dy|x,a).\nonumber
\end{eqnarray}
Using the arguments as in the proof of (\ref{cauchyseq}), we infer that
\begin{eqnarray}
\label{vmncd}
&& |T^{(m-1)}{\bf 0}(y)-T^{(n-1)}{\bf 0}(y)|\le \sup_{\pi\in\Pi}|U_{m-1}(y,\pi)-
U_{n-1}(y,\pi)|\\ 
&\le& \omega(y)\sup_{\pi\in\Pi}\|U_{m-1}(\cdot,\pi)-U_{n-1}(\cdot,\pi)\|_{\omega} \le \omega(y)\tilde{\gamma}^{(n-1)}(\tilde{L}(z)). 
\nonumber 
\end{eqnarray}
From (\ref{vmn}) and (\ref{vmncd}), it follows that
\begin{eqnarray*}
|v_m(a)-v_n(a)|&\le& \int_X \gamma\big(\omega(y)\tilde{\gamma}^{(n-1)}
(\tilde{L}(z))\big)q(dy|x,a)\\ &\le&
\int_X \omega(y)\gamma\big(\tilde{\gamma}^{(n-1)}\big(\tilde{L}(z))\big)q(dy|x,a)\le 
\omega(x)\tilde{\gamma}^{(n)}\big(\tilde{L}(z)\big).
\end{eqnarray*}
for all $a\in A(x).$  
Hence
$$
\max_{a\in A(x)}|v_m(a)-v_n(a)|\le \omega(x)\tilde{\gamma}^{(n)}\big(\tilde{L}(z)\big) =:\varepsilon_n.
$$
 This then implies that $v_m(a) \le v_n(a)+\varepsilon_n$ 
for $m\ge n$ and all $a\in A(x).$ Since $ \varepsilon_n\to 0$ 
as $n\to\infty,$   
the result follows  from Theorem A.1.5. in \cite{br11} and the fact that $v^*=\lim_{n\to \infty} V_n$ (see Remark \ref{valiter}).
 \hfill $\Box$
\end{proof}

\subsection{\bf Howard's policy improvement  algorithm}

The algorithm proposed  by  Howard \cite{howard}
is widely discussed in the literature on Markov decision processes (dynamic programming), see \cite{br11,hll,hl} 
and their references. It may be also applied to models with the recursive discounted utility.
 
For any  $f\in F$ we shall use the following notation $U_f=U(\cdot,f)$. 

\begin{theorem}\label{theo:howard} Let assumptions of Theorem \ref{thm1} be satisfied.
For any $f\in F,$ denote
$$ A(x,f) := \big\{ a\in A(x)\; |\; u(x,a)+\int_X \delta(U(y,f)) q(dy|x,a) > U(x,f)\big\}, \quad
x\in {X}.$$ Then, the following holds:
\begin{enumerate}
\item[(a)] If for some Borel set  ${X}_0\subset {X}$ we define a decision
rule $g$ by
\begin{eqnarray*}
  g(x) \in A(x,f) && \;\mbox{for}\; x\in {X}_0, \\
  g(x) = f(x) && \;\mbox{for}\; x\notin {X}_0,
\end{eqnarray*}
then $U_g\ge U_f$ and $U_g(x) > U_f(x)$ for $x\in {X}_0$. In this case, the policy
$g$ is called {\em an improvement } of $f$.
\item[(b)] If $A(x,f) = \emptyset$ for all $x\in {X}$, then $U_f=v_*$, i.e.,
the stationary policy $f\in F$ is optimal.
\end{enumerate}
\end{theorem}

\begin{proof} 
(a) From the definition of $g$ we obtain $$ T_g U_f (x) > U(x,f),$$ if
  $x\in {X}_0$ and $T_g U_f(x) = U(x,f),$ if $x\notin {X}_0$. Thus by induction
$$ U_f (x)\le T_g U_f(x)\le T_g^{(n)} U_f(x),$$ where the first inequality is strict for $x\in {X}_0$. 
Letting $n\to\infty,$ it follows  as in the proof of Theorem \ref{thm1} that
$U_f \le U_g$ and in particular $U(x,f) < U(x,g)$ for $x\in {X}_0$.
  
(b) The condition $A(x,f) = \emptyset$ for all $x\in {X}$ implies $T U_f \le U_f$. 
Since we always have $T U_f \ge T_f U_f = U_f,$ we obtain
$T U_f = U_f$.   From Theorem \ref{thm1} we know that $T$ 
has a unique  fixed point $v^*\in {\cal U}^d_b(X)$ (under assumptions $(W)$) or $v^*\in {\cal M}^d_b(X)$ (under assumptions $(S)$). 	
Thus $U_f =v^*$. \hfill $\Box$
\end{proof}

Altogether, we have the following algorithm for the computation of the value function and an optimal stationary policy:

\begin{enumerate}
\item Choose $f_0\in F$ arbitrary and set $k=0$.
\item Compute $U_{f_k} $ as the unique solution $v\in {\cal M}^d_b({X})$ of the equation
$v=T_{f_k}v$.
\item Choose $f_{k+1}\in F$ such that 
$$ f_{k+1}(x) \in 
\mbox{Arg}\max_{a\in A(x)} \Big(
u(x,a) + \int_X \delta(U_{f_k}(y))q(dy|x,f_k(x))\Big)$$  
and set $f_{k+1}(x)=f_{k}(x)$ if possible.
If $f_{k+1}=f_{k},$  then $U_{f_{k}}=v^*$ and $(f_k,f_k,\ldots)$ is optimal stationary policy.
Else set $k:=k+1$ and go to   step 2.
\end{enumerate}

It is obvious that the algorithm stops in a finite number of steps
if the state and action sests are finite. In general, as in the standard discounted case
(see, e.g., Theorem 7.5.1 and Corollary 7.5.3 in \cite{br11}) we can only claim that
$$v^*(x)= \lim_{k\to\infty} U(x,f_k).$$

\section{Applications}\label{sec:appl}

\begin{example} {\it (Stochastic optimal growth model 1)}  There is a single good available to the consumer. 
The level of this good at the beginning of period $t\in\mathbb{N}$  is given  by  $x_t\in X:=[0,\infty).$
The consumer has  to divide $x_t$ between consumption $a_t\in A:=[0,\infty)$ 
 and investment (saving) $y_t=x_t-a_t$. Thus $A(x_t) = [0,x_t]$.
From consumption $a_t$ the consumer receives utility $u(x_t,a_t)=\sqrt{a_t}$. 
Investment, on the other hand, is used for production with input $y_t$ yielding output 
$$x_{t+1} = y_t\cdot \xi_t,  \qquad t\in\mathbb{N},$$
 where  $(\xi_t)$ is a sequence of i.i.d. shocks with distribution $\nu$ being a probability measure on $[0,\infty).$ 
 The initial state $x=x_1 \in \mathbb{R}_+$ is non-random. Further, we assume that 
 $$\bar{s} =\mathbb{E}\xi_t =\int_0^\infty s\nu(ds)\le 1.$$
 Let the discount function be as follows
$$\delta(z)=(1-\varepsilon)z+\varepsilon\ln(1+z), \quad z\ge 0$$
with some $\varepsilon\in (0,1).$
We observe that there is no constant  $\beta\in (0,1)$ such that $\delta(z)<\beta z$ for all $z>0.$ 
We define $\gamma:=\delta$ and note that $z\mapsto \gamma(z)/z$ is non-increasing. 
Hence, $\gamma$ and $\delta$ satisfy 
assumptions $(B2.1)$-$(B2.3).$ 
Now we show that assumptions  $(A)$, $(W)$ and  $(B2.4)$ 
are satisfied with an appropriate function $\omega.$ 
With this end in view, put 
$$\omega(x)=\sqrt{x+1}, \quad x\in X.$$
Then $|u(x,a)|\le \sqrt{x}\le\sqrt{x+1}$ for $a\in A(x)=[0,x]$ and $x\in X.$ Thus, $(A)$ holds. 
Furthermore, by Jensen's inequality we have
$$\int_X \omega(y)q(dy|x,a)=\int_0^\infty \sqrt{(x-a)s+1}\ \nu(ds)\le\sqrt{ (x-a)\bar{s}+1}\le \sqrt{x+1}$$
for $a\in A(x)$ and $x\in X.$
Hence, $(B2.4)$ is satisfied with $\alpha=1.$ It is   obvious that conditions $(W)$
are also met. Therefore, from Theorem \ref{thm1}, there exists an upper semicontinuous  solution 
to the Bellman equation and a stationary optimal policy $f^*\in F.$
\end{example}

\begin{example}   {\it (Stochastic optimal growth model 2)}  We study a modified problem model from Example 1. We assume that
the next state evolves according to the equation
$$x_{t+1} = y_t^\theta\cdot \xi_t+(1-\rho)\cdot y_t,\qquad t\in\mathbb{N}$$
where  $\rho$, $\theta$  are some constants from the interval $(0,1).$ Here,  $(\xi_t)$
 is again a sequence of i.i.d. shocks with distribution 
$\nu$ and the expected value $\bar{s}>0$. 
The utility  function is now $u(x,a)=a^\sigma$ with $\sigma\in (0,1).$ Let  $\omega(x)=(x+r)^\sigma,$ where $r\ge 1.$
Assume now that the discount  function is of the form
$$\delta(z)=(1-2\varepsilon)z+\varepsilon\ln(1+z), \quad z\ge 0$$
with some $\varepsilon\in (0,1).$ Then, $\delta(z)\le(1-\varepsilon)z$ for $z\ge 0.$ By 
Example 1 in \cite{jet} (see also Example 2 in \cite{jano}), we have that
$$\int_X \omega(y)q(dy|x,a)\le \left(1+\frac{(\bar{s}/\rho) ^{\frac 1{1-\theta}}}r\right)^\sigma\omega(x),\quad x\in X.$$
Hence, in $(B2.4)$, we set
$$\alpha:= \left(1+\frac{(\bar{s}/\rho)^{\frac 1{1-\theta}}}r\right)^\sigma>1$$
and $\alpha \gamma(z)< z$ is satisfied with $\gamma:=\delta$ and   $r$ such that $\alpha(1-\varepsilon)<1.$
Clearly, all conditions $(A)$, $(B)$ and $(W)$   are satisfied. 
($B(2.3) $  follows from the fact that $z\mapsto\gamma(z)/z$ is non-increasing.)
By Theorem  \ref{thm1}, the value function $v^*$ is upper semicontinuous and satisfies the Bellman equation.
Moreover, there exists an optimal stationary policy.
\end{example}

\begin{example}  {\it (Inventory model)} A manager inspects the stock at each period 
$t\in\mathbb{N}.$ 
The number of units in stock is $x_t\in X:=[0,\infty).$ 
He can sell $\min\{x_t,\Delta_t\}$ units in period $t$, 
where $\Delta_t\ge 0$ is a   random variable representing an unknown demand. 
At the end of period $t$ he can also order any amount 
$a_t\in A:=[0,\hat{a}]$  of new goods to be delivered 
at the beginning of next period at a cost $C(a_t)$ paid immediately. 
Here $\hat{a}$ is some positive constant.
Moreover, the function $C$ is bounded above by $\hat{C}$, lower semicontinuous,  
non-f and $C(0)=0.$ 
The state equation is of the form
$$x_{t+1}=x_t-\min\{x_t,\Delta_t\}+a_t,\quad\mbox{for } t\in\mathbb{N},$$
where $(\Delta_t)$ is a sequence of i.i.d. random variables such that each $\Delta_t$ follows a 
continuous distribution
$\Phi$ and $\mathbb{E}\Delta_t<\infty.$ 
The manager considers a recursive discounted utility with a discount  function $\delta$ 
satisfying $(B2.1)$-$(B2.3)$  (with $\omega =1$). 
This model can be viewed as a Markov decision process, in which 
$u(x,a):=p\mathbb{E}\min\{x,\Delta\}-C(a)$ is the one period bounded utility function and $p$ denotes the unit stock price.  
(Here  $\Delta= \Delta_t$ for any fixed $t\in\mathbb{N}.$) 
Clearly, 
$ -\hat{C}\le u(x,a)\le p\mathbb{E}\Delta.$ Next note that
the transition probability $q$ is of the form
$$
q(B|x,a)=\int_{0}^{\infty}1_{B}(x-\min\{x,y\}+a)dF(y),
$$
where $B\subset X$ is a 
Borel set, $x\in X$, $a\in A.$ 
If $\phi$ is a bounded continuous function on $X$, then the integral
\begin{eqnarray*}
\int_{X}\phi(y)q(dy|x,a)&=&\int_{0}^{\infty}\phi(x-\min\{x,y\}+a)dF(y )\\
&=&\int_{0}^{x}\phi(x-y+a)dF(y)+
\int_{x}^{\infty}\phi(a)dF(y)\\
&=&\int_{0}^{x}\phi(x-y+a)dF(y)+
\phi(a)(1-F(x))
\end{eqnarray*}
depends continuously on $(x,a).$
Hence, the model satisfies assumptions $(W2.1)$-$(W2.3)$. Therefore, by Theorem \ref{thm1},   
there exists a bounded upper semicontinuous solution to the Bellman  equation and an optimal  stationary policy $f^*\in F.$
\end{example}

\begin{example}  {\it (Stopping problem)}
We now decribe a stopping problem with non-linear discounting. 
Suppose the Borel state space is $X$  and there is an (uncontrolled) Markov chain with an initial distribution $q_0$ 
and the transition probability  $q(\cdot|x).$ By $\mathbb{P}$ we denote the probability measure on the product space 
$X^\infty$ of all trajectories of the Markov chain induced by $q_0$ and $q.$  
At each period the controller has to decide whether to stop the process and receive the reward $R(x),$ where 
$x$ is the current state or to continue. In the latter case the reward $C(x)$ (which might be a cost) is received. 
The aim is to find a stopping time such that the recursive  discounted reward is maximized. 
We assume here that the controller has to stop with probability one. 
This problem is a special case of the more general model of Section 2. We have to choose here
\begin{itemize}
  \item[(i)] $X\cup \{\infty\}$ is the state space where $\infty$ is an  absorbing state which indicates that the process is already stopped,
  \item[(ii)] $A:=\{0,1\}$ where $a=0$ means continue and $a= 1$ means
  stop,
  \item[(iii)] $A(x):= A$ for all $x\in X\cup \{\infty\}$,
  \item[(iv)] ${q}(B | x,0):=q(B|x)$ for $x\in X,$ $ B$ a Borel set and $q(\{\infty\} | x,1)=1$ for $x\in X$ and $q(\infty|\infty,\cdot)=1$,
  \item[(v)] ${u}(x,a) := C(x)(1-a)+ R(x) a $ for $x\in X$ and $u(\infty,\cdot)=0$.
 \end{itemize}
We assume now that $|C(x)|\le \omega(x)$ and $|R(x)|\le \omega(x)$ which implies $(A)$ and we assume $(B).$  
The   optimisation problem   is considered 
with the recursive discounted utility and an  interpretation that 
the receiving  of   rewards (costs) is stopped after a random time. This random time is a stopping time 
with respect to the  filtration generated by observable states (for more details see Chapter 10 in \cite{br11}).
\end{example}

\begin{proposition}\label{thm:optimalstop} If the above 
assumptions on the stopping problem are satisfied,  
then \\\noindent
(a) there exists a  function $v^*\in {\cal M}^d_b(X)$   such that
\begin{equation}\label{oeu} 
v^*(x)=\max\left\{R(x); \; C(x) +\int_X \delta(v^*(y))q(dy|x)\right\}.
\end{equation}
(b)  Moreover, define $f^*(x)=1$ if $v^*(x)=R(x)$ and $f^*(x)=0$
 else and let $\tau^* := \inf\{n\in\mathbb{N} : f^*(x_n)=1\}$. If $\mathbb{P}(\tau^*<\infty)=1,$ then $\tau^*$  is
  optimal for the stopping problem and $v^*(x)=\sup_{\pi\in\Pi}U(x,\pi).$
\end{proposition}

 The action space $A$ consists of two elements only, 
so assumptions $(S)$ are  satisfied. In the above description we already
assumed $(A)$ and $(B).$ 
Therefore the result follows from Theorem \ref{thm1}

Let us consider a special example. 
Imagine a person who wants to sell her house. At the beginning of
each week she receives an offer, which is randomly distributed over
the interval $[m,M]$  with $0<m<M$. The offers are independent and
identically distributed with distribution $q$. The house seller has
to decide immediately whether to accept or reject this offer. If she
rejects, the offer is lost and she has maintenance cost. Which offer should she
accept in order to maximise her expected reward?

Here, we have $X:=[m,M]$,  $C(x)\equiv -c$  and $R(x):=x$. 
From Proposition \ref{thm:optimalstop} we obtain that the value function satisfies
$$v^*(x)=\max\left\{x; -C +\int_{[m,M]} \delta(v_*(y))q(dy)\right\}.$$ 
Note that $C^*:=\int_{[m,M]} \delta(v^*(y))q(dy)$
 is obviously a constant independent of $x$. Thus the optimal strategy is to accept the first offer which is above $-C+C^*$. 
 The corresponding stopping time is geometrically distributed and thus certainly satisfies  $\mathbb{P}(\tau^*<\infty)=1$. 
 Moreover, it is not difficult to see that whenever we have two discount functions $\delta_1,\delta_2,$ 
 which satisfy assumptions $(B)$ and which are ordered, i.e.,  $\delta_1\le \delta_2$, then $C_1^* \le C_2^*$  because the operator $T$
 is monotone. Thus, with stricter discounting we will stop earlier.

\end{document}